\documentclass[11pt,a4paper]{article}
\usepackage[pdftex]{graphicx}
\usepackage{amsmath,amssymb,amsfonts,amsthm}
\numberwithin{equation}{section}
\usepackage{indentfirst}
\usepackage{enumitem} 
\usepackage[margin=1.3in]{geometry}
\usepackage{fancyhdr}

\usepackage[colorlinks=true,linkcolor=magenta,citecolor=blue,urlcolor=cyan]{hyperref}
\usepackage{titlesec}
\usepackage{etoolbox}
\usepackage{graphicx}  
\usepackage{changepage}
\theoremstyle{plain}
\newtheorem{theorem}{Theorem}[section]
\newtheorem{lemma}[theorem]{Lemma}

\newtheorem{remark}{Remark}[section]

\makeatletter
\renewcommand{\maketitle}{
	\begin{center}
		{\Large\bfseries{\@title}\par}
		\vskip 1em
		{\normalsize
			\lineskip .5em
			\begin{tabular}[t]{c}
				\@author
			\end{tabular}\par}
		\vskip 1.5em
	\end{center}
}
\makeatother
\renewenvironment{abstract}{
	\begin{adjustwidth}{1.3cm}{1.3cm}
		\noindent{\large\bfseries{A{\scriptsize BSTRACT.}}}
	}{
	\end{adjustwidth}
}

\usepackage{color}
\usepackage{xcolor}
\usepackage[normalem]{ulem} 
\usepackage{soul}

\usepackage{graphicx} 
\usepackage{xstring}  

\usepackage{xstring}
\usepackage{graphicx} 

\usepackage{marvosym}

\begin{document}
	
	\title{Generalization of Ramanujan's Continued Fractions for Even Order}

	\author{Dipika Sarkar, S. N. Fathima, and M. P. Thejitha}
	
	\maketitle
	
	\begin{abstract}
		In this paper, we derive three generalized continued fractions of any even order $k$ with the aid of a general continued fraction identity of Ramanujan and we establish general theta function identities for these continued fractions. As an application of continued fraction of order seventy-six, we obtain partition theoretic identities and some vanishing coefficient results. 
		\\
		
		
	\end{abstract}
	
	\bigskip
	
	\vspace{0.5em}
	\section{Introduction} 
	
	For any complex numbers $\lambda$ and $q$, define the $q$-product $(\lambda; q)_\infty$ as
	\begin{equation*}
		(\lambda; q)_\infty := \prod_{n=0}^{\infty} \left(1 - \lambda q^n \right), \quad |q| < 1.
	\end{equation*}
		For brevity, we often write
	\[
	(\lambda_1; q)_\infty (\lambda_2; q)_\infty (\lambda_3; q)_\infty \cdots (\lambda_m; q)_\infty
	= (\lambda_1, \lambda_2, \lambda_3, \cdots, \lambda_m; q)_\infty.
	\]
	\indent Ramanujan was a pioneer in the field of $q$-continued fractions and he recorded several continued fraction identities in his notebooks. The most famous among them 
	is the Rogers--Ramanujan continued fraction $R(q)$ defined by
	\begin{equation}
		R(q) := q^{1/5}\frac{(q,q^{4};q^{5})_{\infty}}{(q^{2},q^{3};q^{5})_{\infty}}
		= q^{1/5}\frac{f(-q,-q^{4})}{f(-q^{2},-q^{3})}
		= \cfrac{q^{1/5}}{1+\cfrac{q}{1+\cfrac{q^{2}}{1+\cfrac{q^{3}}{1+\cdots}}}},
		\qquad |q|<1.
		\label{a1.1}
	\end{equation}
	The Rogers--Ramanujan continued fraction $R(q)$ is often referred to as the 
	continued fraction of order five. In continuation to his work, he obtained some theta-function identities and modular relations for the continued fraction $R(q)$, see \cite{Berndt}. Ramanujan also recorded the following general continued fraction 
	identity \cite[p.~24, Entry~12]{Berndt}:
	Suppose that $a$, $b$, and $q$ are complex numbers with $|ab|<1$, 
	or for some integer $m$, $a=bq^{2m+1}$. Then	
	\begin{equation}
		\frac{(a^{2}q^{3};q^{4})_{\infty}(b^{2}q^{3};q^{4})_{\infty}}
		{(a^{2}q;q^{4})_{\infty}(b^{2}q;q^{4})_{\infty}}
		=
		\cfrac{1}{1-ab+
			\cfrac{(a-bq)(b-aq)}
			{(1-ab)(q^{2}+1)+
				\cfrac{(a-bq^{3})(b-aq^{3})}
				{(1-ab)(q^{4}+1)+\cdots}}
		}.
		\label{a1.2}
	\end{equation}
	For specific values of $a$ and $b$, and suitable powers of $q$, 
	one can obtain $q$-continued fractions of particular order which satisfy 
	theta-function identities analogous to those of $R(q)$. The main purpose of this paper is to establish the generalized form of $q$-continued fractions of any even order $k$ and derive theta-function identities analogous to those of $R(q)$.\\
	\indent We now replace $q$ by $q^{\frac{k}{4}}$ and then we set $(a, b)= {\left(q^{2i-1}, q^{\frac{k}{4}-(2i-1)} \right)}$ in \eqref{a1.1}, where $ i=\frac{m}{2}$ and $m\in\mathbb{N}$, and using
	\begin{align*}
		\left(q^{\frac{5k}{4} - (4i - 2)}; q^{k}\right)_{\infty}
		=
		\frac{\left
			(q^{\frac{k}{4} - (4i - 2)}; q^{k}\right)_{\infty}
		}{\left(1 - q^{\frac{k}{4} - (4i - 2)}\right)},
	\end{align*}
	in the resulting identity to obtain the following continued fraction ${S_n(q)}$ of order ${k}$:
		\begin{align}
			S_n(q) 
			:&= q^{(2i - 1)}
			\frac{
				\left(q^{\frac{k}{4} - (4i - 2)},\; q^{\frac{3k}{4} + (4i - 2)}; q^{k}\right)_{\infty}
			}{
				\left(q^{\frac{k}{4} + (4i - 2)},\; q^{\frac{3k}{4} - (4i - 2)}; q^{k}\right)_{\infty}
			} 
			= q^{(2i-1)}
			\frac{
				f\left(-q^{\frac{k}{4}-(4i-2)}, -q^{\frac{3k}{4}+(4i-2)}\right)
			}{f \left(-q^{\frac{k}{4}+(4i-2)}, -q^{\frac{3k}{4}-(4i-2)}\right)} \notag \\[6pt]
			&=
			\cfrac{q^{(2i-1)}\left(1 - q^{\frac{k}{4}-(4i-2)}\right)}
			{\left(1 - q^{\frac{k}{4}}\right)
				+ \cfrac{
					q^{\frac{k}{4}}\left(1 - q^{4i-2}\right)\left(1 - q^{\frac{k}{2}-(4i-2)}\right)
				}{
					\left(1 - q^{\frac{k}{4}}\right)\left(1 + q^{\frac{k}{2}}\right)
					+ \cfrac{
						q^{\frac{k}{4}}\left(1 - q^{\frac{k}{2}+(4i-2)}\right)\left(1 - q^{k-(4i-2)}\right)
					}{
						\left(1 - q^{\frac{k}{4}}\right)\left(1 + q^{k}\right) + \cdots
					}
				}
			}.\label{a1.3}
			\end{align}
	Similarly, we replace $q$ by $q^{\frac{k}{4}}$ and then we set $(a, b)= {\left(q^{\frac{(2i-1)}{2}}, q^{\frac{k}{4}-\frac{(2i-1)}{2}}\right)}$ in \eqref{a1.1}, where $ i=\frac{m}{2}$ and $m\in\mathbb{N}$, and using
	\begin{align*}
		\left
		(q^{\frac{5k}{4}-(2i-1)}; q^{k}\right)_{\infty}
		=
		\frac{\left(q^{\frac{k}{4} - (2i - 1)}; q^{k}\right)_{\infty}
		}{\left(1 - q^{\frac{k}{4} - (2i - 1)}\right)},
	\end{align*}
	in the resulting identity to obtain the following continued fraction ${T_n(q)}$ of order ${k}$:
		\begin{align*}
			T_n(q) 
			:&= q^{\left(\frac{2i - 1}{2}\right)}
			\frac{
				\left(q^{\frac{k}{4} - (2i - 1)},\; q^{\frac{3k}{4} + (2i - 1)}; q^{k}\right)_{\infty}
			}{
				\left(q^{\frac{k}{4} + (2i - 1)},\; q^{\frac{3k}{4} - (2i - 1)}; q^{k}\right)_{\infty}
			} 
			= q^{\left(\frac{2i - 1}{2}\right)}
			\frac{
				f\left(-q^{\frac{k}{4}-(2i-1)}, -q^{\frac{3k}{4}+(2i-1)}\right)
			}{
				f\left(-q^{\frac{k}{4}+(2i-1)}, -q^{\frac{3k}{4}-(2i-1)}\right)
			} 
			\end{align*}
			\begin{align}
			&=
			\cfrac{q^{\left(\frac{2i - 1}{2}\right)}\left(1 - q^{\frac{k}{4}-(2i-1)}\right)}
			{\left(1 - q^{\frac{k}{4}}\right)
				+ \cfrac{
					q^{\frac{k}{4}}\left(1 - q^{2i-1}\right)\left(1 - q^{\frac{k}{2}-(2i-1)}\right)
				}{
					\left(1 - q^{\frac{k}{4}}\right)\left(1 + q^{\frac{k}{2}}\right)
					+ \cfrac{
						q^{\frac{k}{4}}\left(1 - q^{\frac{k}{2}+(2i-1)}\right)\left(1 - q^{k-(2i-1)}\right)
					}{
						\left(1 - q^{\frac{k}{4}}\right)\left(1 + q^{k}\right) + \cdots
					}
				}
			}.
			\label{a1.4}
		\end{align}
	Again, we replace $q$ by $q^{\frac{k}{4}}$ and then we set  $(a, b)= {\left(q^{\frac{(2i-1)}{4}}, q^{\frac{k}{4}-\frac{(2i-1)}{4}}\right)}$ in \eqref{a1.1}, where $ i=\frac{m}{2}$ and $m\in\mathbb{N}$, and using
	\begin{align*}
		\left(q^{\frac{5k-(4i-2)}{4}}; q^{k}\right)_{\infty}
		=
		\frac{\left(q^{\frac{k-(4i - 2)}{4}}; q^{k}\right)_{\infty}}
		{\left(1 - q^{\frac{k-(4i - 2)}{4}}\right)},
	\end{align*}
in the resulting identity to obtain the following continued fraction ${U_n(q)}$ of order ${k}$:
		\begin{align}
			U_n(q) 
			:&= q^{\left(\frac{2i - 1}{4}\right)}
			\frac{
				\left(q^{\frac{k - (4i - 2)}{4}},\; q^{\frac{3k + (4i - 2)}{4}}; q^{k}\right)_{\infty}
			}{
				\left(q^{\frac{k + (4i - 2)}{4}},\; q^{\frac{3k - (4i - 2)}{4}}; q^{k}\right)_{\infty}
			}
			= q^{\left(\frac{2i - 1}{4}\right)}
			\frac{
				f\left(-q^{\frac{k - (4i - 2)}{4}}, -q^{\frac{3k + (4i - 2)}{4}}\right)
			}{
				f\left(-q^{\frac{k + (4i - 2)}{4}}, -q^{\frac{3k - (4i - 2)}{4}}\right)
			} \notag \\[6pt]
			&=
			\cfrac{q^{\left(\frac{2i-1}{4}\right)}\left(1 - q^{\frac{k-(4i-2)}{4}}\right)}
			{(1 - q^{\frac{k}{4}})
				+ \cfrac{
					q^{\frac{k}{4}}\left(1 - q^{\frac{(2i-1)}{2}}\right)(1 - q^{\frac{k-(2i-1)}{2}})}{
					\left(1 - q^{\frac{k}{4}}\right)\left(1 + q^{\frac{k}{2}}\right)
					+ \cfrac{q^{\frac{k}{4}}\left(1 - q^{\frac{k+(2i-1)}{2}}\right)\left(1 - q^{k-{\frac{(2i-1)}{2}}}\right)}{\left(1 - q^{\frac{k}{4}}\right)\left(1 + q^{k}\right) + \cdots}
				}
			}.
			\label{a1.5}
		\end{align}	

	\indent In this paper, we establish generalised formulas of theta-function for the continued fractions $S_n(q)$, $T_n(q)$ and $U_n(q)$. The following are our main results.
	
	\begin{theorem}\label{at1.1}
		Let $S_n(q)$ be defined as in \eqref{a1.3}. We have
		\begin{align}
			\frac{1}{S_n(q)} - S_n(q)
			&= \frac{\phi(q^{\frac{k}{4}})\, f\left(-q^{(4i-2)}, -q^{\frac{k}{2}-(4i-2)}\right)}
			{q^{(2i-1)}\psi(q^{\frac{k}{2}})\, f\left(-q^{\frac{k}{4}-(4i-2)},-q^{{\frac{k}{4}+(4i-2)}}\right)}\label{a1.6}\\
			\frac{1}{S_n(q)} + S_n(q)
			&= \frac{\phi(-q^{\frac{k}{4}})\, f\left(q^{(4i-2)}, q^{\frac{k}{2}-(4i-2)}\right)}
			{q^{(2i-1)}\psi(q^{\frac{k}{2}})\, f\left(-q^{\frac{k}{4}-(4i-2)},-q^{{\frac{k}{4}+(4i-2)}}\right)}.\label{a1.7}	
		\end{align}
	\end{theorem}
	
	\begin{theorem}\label{at1.2}
		Let $T_n(q)$ be defined as in \eqref{a1.4}. We have
		\begin{align}
		\frac{1}{T_n(q)} - T_n(q)
		&= \frac{\phi(q^{\frac{k}{4}})\, f\left(-q^{(2i-1)}, -q^{\frac{k}{2}-(2i-1)}\right)}
		{q^{\left(\frac{2i-1}{2}\right)}\psi(q^{\frac{k}{2}})\, f\left(-q^{\frac{k}{4}-(2i-1)},-q^{\frac{k}{4}+(2i-1)}\right)}\label{a1.8}\\
		\frac{1}{T_n(q)} + T_n(q)
		&= \frac{\phi(-q^{\frac{k}{4}})\, f\left(q^{(2i-1)}, q^{\frac{k}{2}-(2i-1)}\right)}
		{q^{\left(\frac{2i-1}{2}\right)}\psi(q^{\frac{k}{2}})\, f\left(-q^{\frac{k}{4}-(2i-1)},-q^{\frac{k}{4}+(2i-1)}\right)}.\label{a1.9}
		\end{align}
	\end{theorem} 
	
	\begin{theorem}\label{at1.3}
		Let $U_n(q)$ be defined as in \eqref{a1.5}. We have
		\begin{align}
			\frac{1}{U_n(q)} - U_n(q)
			&= \frac{\phi(q^{\frac{k}{4}})\, f\left(-q^{\frac{(2i-1)}{2}},  -q^{\frac{k-(2i-1)}{2}}\right)}
			{q^{\left(\frac{2i-1}{4}\right)}\psi(q^{\frac{k}{2}})\, f\left(-q^{\frac{k-(4i-2)}{4}},-q^{\frac{k+(4i-2)}{4}}\right)}\label{a1.10}\\
			\frac{1}{U_n(q)} + U_n(q)
			&= \frac{\phi(-q^{\frac{k}{4}})\, f\left(q^{\frac{(2i-1)}{2}}, q^{\frac{k-(2i-1)}{2}}\right)}
			{q^{\left(\frac{2i-1}{4}\right)}\psi(q^{\frac{k}{2}})\, f\left(-q^{\frac{k-(4i-2)}{4}},-q^{\frac{k+(4i-2)}{4}}\right)}.\label{a1.11}
		\end{align}
	\end{theorem} 
	\indent The rest of the paper is organized as follows. In Section \ref{as2}, we recall some theta functions identities essential for the proof of Theorem \ref{at1.1}-\ref{at1.3} in Section \ref{as3}. In Section \ref{as4}, we obtain theta function identities for continued fraction of order seventy-six from our main theorems. As a result, we prove color partition identities in Section \ref{as5} and conclude with few vanishing coefficient results.
	
	\section{Preliminaries}\label{as2}
	In this section, we revisit theta functions identities which will be employed in the subsequent sections. The Ramanujan’s general theta-function $f(a,b)$ \cite[p.~34]{Berndt} is defined as
	\begin{equation*}
		f(a,b) = \sum_{n=-\infty}^{\infty} a^{\frac{n(n+1)}{2}} b^{\frac{n(n-1)}{2}}, 
	\quad |ab| < 1.
	\end{equation*}
Jacobi’s triple product identity \cite[p.~35, Entry 19]{Berndt} is arguably the most celebrated theorem in the theory of theta functions and can be stated in terms of $f(a,b)$, as
	\begin{equation}\label{a2}
		f(a,b) = (-a;ab)_\infty (-b;ab)_\infty (ab;ab)_\infty
		= (-a, -b, ab; ab)_\infty.
	\end{equation}
 The three most important special cases of \eqref{a2} are defined by \cite[p.~36, Entry 22]{Berndt}
	\begin{align}
		\phi(q) &:= f(q,q) = \sum_{n=-\infty}^{\infty} q^{n^2}
		= \frac{(-q; -q)_\infty}{(q; -q)_\infty}\label{a2.1}
	\end{align}
	\begin{align}
		\psi(q) &:= f(q,q^3) = \sum_{n=0}^{\infty} q^{\frac{n(n+1)}{2}}
		= \frac{(q^2; q^2)_\infty}{(q; q^2)_\infty} \label{a2.2}
	\end{align}
	\begin{align}
		f(-q) &:= f(-q,-q^2) = \sum_{n=-\infty}^{\infty} (-1)^n q^{\frac{n(3n-1)}{2}}
		= (q;q)_\infty.\label{a2.3}
	\end{align}
	Ramanujan also defined the function $\chi(q)$ \cite[p.~36, Entry 22(iv)]{Berndt} as
	\begin{equation*}
		\chi(q) = (-q; q^2)_\infty.
	\end{equation*}
	\indent To prove our main results we provide the following lemmas.
	\begin{lemma}\label{al2.1}
		We have the following:
		\begin{equation*}
				f(a,b)=f(a^3b,ab^3)+af(b/a,a^5b^3).
		\end{equation*}
	\end{lemma}

	\begin{proof}
		From \cite[p.~46]{Berndt}, Entry 30 (ii) and (iii), we have
		\begin{equation}
			f(a,b)+ f(-a,-b)=2f(a^3b,ab^3)
			\label{a2.4}
		\end{equation}
		and
		\begin{equation}
			f(a,b) - f(-a,-b)=2af(b/a,a^5b^3).
			\label{a2.5}
		\end{equation}
		By suitably manipulating equations \eqref{a2.4} and \eqref{a2.5}, the Lemma \ref{al2.1} follows.
	\end{proof}
	\begin{lemma}\label{al2.2}{\cite[p.~46, Entry~30(ii)-(iii)]{Berndt}}
	We have
		\begin{align}
			f(a,ab^2)f(b,a^2b) &=f(a,b)\psi(ab)\label{a2.6}\\
			f(a,b)f(-a,-b) &=f(-a^2,-b^2)\phi(-ab).\label{a2.7}
		\end{align}
	\end{lemma}
	\begin{lemma}\label{al2.3}{{\cite[p.~46, Entry~30(v)-(vi)]{Berndt}}}
		We have
		\begin{equation*}
			f^2(a,b)=f(a^2,b^2)\phi(ab)+2af(b/a,a^3b)\psi(a^2b^2).
		\end{equation*}
	\end{lemma}

	\section{Proof of Theorem \ref{at1.1}-\ref{at1.3}} \label{as3}
	In this section, keeping in mind the prerequisites of Section \ref{as2}, we offer the proof of Theorem \ref{at1.1}.
	\begin{proof}[Proof of Theorem \ref{at1.1}] 
	From \eqref{a1.3}, we see that
		\begin{equation}
			\frac{1}{\sqrt{S_n(q)}} - \sqrt{S_n(q)}
			= 
			\frac{f(-q^{\frac{k}{4}+(4i-2)}, -q^{\frac{3k}{4}-(4i-2)})-q^{(2i-1)}f(-q^{\frac{k}{4}-(4i-2)}, -q^{\frac{3k}{4}+(4i-2)})}
			{\sqrt{{q^{(2i-1)}\, f(-q^{\frac{k}{4}-(4i-2)},-q^{{\frac{3k}{4}+(4i-2)}})} f(-q^{\frac{k}{4}+(4i-2)}, -q^{\frac{3k}{4}-(4i-2)})}}.
			\label{a3.1}
		\end{equation}
	Using Lemma \ref{al2.1} with $(a, b) = {\left(-q^{(2i-1)}, q^{\frac{k}{4}-(2i-1)} \right)}$ and ${\left(q^{(2i-1)}, -q^{\frac{k}{4}-(2i-1)}\right)}$, we obtain
	\begin{equation}
		f(-q^{(2i-1)},q^{\frac{k}{4}-(2i-1)})
		=
		f(-q^{\frac{k}{4}+(4i-2)},-q^{\frac{3k}{4}-(4i-2)})-q^{(2i-1)}f(-q^{\frac{k}{4}-(4i-2)},-q^{\frac{3k}{4}+(4i-2)})
		\label{a3.2}
	\end{equation}
	and
	\begin{equation}			
		f(q^{(2i-1)},-q^{\frac{k}{4}-(2i-1)})
		=
		f(-q^{\frac{k}{4}+(4i-2)},-q^{\frac{3k}{4}-(4i-2)})+q^{(2i-1)}f(-q^{\frac{k}{4}-(4i-2)},-q^{\frac{3k}{4}+(4i-2)}),
		\label{a3.3}
	\end{equation}
	respectively.\\
	Applying \eqref{a3.2} in \eqref{a3.1}, we find that
	\begin{equation}
		\frac{1}{\sqrt{S_n(q)}} - \sqrt{S_n(q)}
		= \frac{f(-q^{(2i-1)},q^{\frac{k}{4}-(2i-1)})}
		{\sqrt{{q^{(2i-1)}\, f(-q^{\frac{k}{4}-(4i-2)},-q^{{\frac{3k}{4}+(4i-2)}})} f(-q^{\frac{k}{4}+(4i-2)}, -q^{\frac{3k}{4}-(4i-2)})}}.
		\label{a3.4}
	\end{equation}
	Thanks to \eqref{a1.3} and \eqref{a3.3}, we obtain	
	\begin{equation}
		\frac{1}{\sqrt{S_n(q)}} + \sqrt{S_n(q)}
		= \frac{f(q^{(2i-1)},-q^{\frac{k}{4}-(2i-1)})}
		{\sqrt{{q^{(2i-1)}\, f(-q^{\frac{k}{4}-(4i-2)},-q^{{\frac{3k}{4}+(4i-2)}})} f(-q^{\frac{k}{4}+(4i-2)}, -q^{\frac{3k}{4}-(4i-2)})}}.
		\label{a3.5}
	\end{equation}
	Now, from \eqref{a3.4} and \eqref{a3.5}, we have
	\begin{equation}
		\frac{1}{S_n(q)}-S_n(q)
		=
		\frac{f(-q^{(2i-1)},q^{\frac{k}{4}-(2i-1)})\, f(q^{(2i-1)},-q^{\frac{k}{4}-(2i-1)})}{{q^{(2i-1)}\, f(-q^{\frac{k}{4}-(4i-2)},-q^{{\frac{3k}{4}+(4i-2)}})} f(-q^{\frac{k}{4}+(4i-2)}, -q^{\frac{3k}{4}-(4i-2)})}.
		\label{a3.6}
	\end{equation}
	Again, we set $(a, b)= {\left(-q^{\frac{k}{4}-(4i-2)}, -q^{\frac{k}{4}+(4i-2)}\right)}$ and ${\left(-q^{(2i-1)},q^{\frac{k}{4}-(2i-1)}\right)}$ in \eqref{a2.6} and \eqref{a2.7}, to obtain
	\begin{equation}
		f(-q^{\frac{k}{4}-(4i-2)},-q^{\frac{3k}{4}+(4i-2)})f(-q^{\frac{k}{4}+(4i-2)},-q^{\frac{3k}{4}-(4i-2)})
		=
		f(-q^{\frac{k}{4}-(4i-2)},-q^{\frac{k}{4}+(4i-2)})\psi(q^{\frac{k}{2}})
		\label{a3.7}
	\end{equation}
	and 
	\begin{equation}
		f(-q^{(2i-1)},q^{\frac{k}{4}-(2i-1)})f(q^{(2i-1)},-q^{\frac{k}{4}-(2i-1)})=f(-q^{4i-2},-q^{\frac{k}{2}-(4i-2)})\phi(q^{\frac{k}{4}}),
		\label{a3.8}
	\end{equation}
	respectively. Employing \eqref{a3.7} and \eqref{a3.8} in \eqref{a3.6}, we complete the proof of \eqref{a1.6}.\\
	\indent Next, we square both the sides of \eqref{a3.5}, to obtain
	\begin{equation}
		\frac{1}{S_n(q)}+S_n(q)
		=
		\frac{f^2(-q^{(2i-1)},q^{\frac{k}{4}-(2i-1)})}
		{q^{(2i-1)}\, f(-q^{\frac{k}{4}-(4i-2)},-q^{{\frac{3k}{4}+(4i-2)}}) f(-q^{\frac{k}{4}+(4i-2)}, -q^{\frac{3k}{4}-(4i-2)})}-2.
		\label{a3.9}
	\end{equation}
	Using Lemma \ref{al2.3} with $(a, b)= {\left(-q^{(2i-1)},q^{\frac{k}{4}-(2i-1)}\right)}$, we obtain
	\begin{align}
		f^2(-q^{(2i-1)}, q^{\frac{k}{4}-(2i-1)})=&f(q^{(4i-2)},q^{\frac{k}{2}-(4i-2)})\phi(-q^{\frac{k}{4}})\nonumber\\
		&+2q^{(2i-1)}f(-q^{\frac{k}{4}-(4i-2)},-q^{\frac{k}{4}+(4i-2)})\psi(q^{\frac{k}{2}}).                                                    
		\label{a3.10}
	\end{align}
	Employing \eqref{a3.7} and \eqref{a3.10} in \eqref{a3.9} and simplifying, we arrive at \eqref{a1.7}. This completes the proof of Theorem \ref{at1.1}.
	\end{proof}
	\begin{proof}[Proof of Theorem \ref{at1.2}-\ref{at1.3}]
		The proof of the Theorems \ref{at1.2} and \ref{at1.3} are similar to the proof of Theorem \ref{at1.1}, so we omit the details.
	\end{proof}
	\begin{remark}
	As $k$ denotes the order of the identity, suitable choices of $k$ in the generalized formula results in continued fractions of different even orders established in \cite{Saikia.1, Saikia.4, Saikia.2, Saikia.3, Raksha, D.S}.	
	\end{remark}
	\section{Application of the Generalized Formula} \label{as4}
	 In this section, we establish theta-function  identities for the continued fractions $Y_n(q)$ of order seventy-six, where $n=1, 2, 3, 4, 5, 6, 7, 8$ and $9$.
	\begin{theorem}\label{at4.1}
		We have
		\begin{align} 
		\frac{1}{Y_1(q)} \pm Y_1(q)
		&= \frac{\phi(\pm q^{19})\, f(\pm q^{2}, \pm q^{36})}
		{q\psi(q^{38})\, f(-q^{17},-q^{21})},\label{a4.1}\\
		\frac{1}{Y_2(q)} \pm Y_2(q)
		&= \frac{\phi(\pm q^{19})\, f(\pm q^{4}, \pm q^{34})}
		{q^{2}\psi(q^{38})\, f(-q^{15},-q^{23})},\label{a4.2}\\
		\frac{1}{Y_3(q)} \pm Y_3(q)
		&= \frac{\phi(\pm q^{19})\, f(\pm q^{6}, \pm q^{32})}
		{q^{3}\psi(q^{38})\, f(-q^{13},-q^{25})},\label{a4.3}\\ 
		\frac{1}{Y_4(q)} \pm Y_4(q)
		&= \frac{\phi(\pm q^{19})\, f(\pm q^{8}, \pm q^{30})}
		{q^{4}\psi(q^{38})\, f(-q^{11},-q^{27})},\label{a4.4}\\ 
		\frac{1}{Y_5(q)} \pm Y_5(q)
		&= \frac{\phi(\pm q^{19})\, f(\pm q^{10}, \pm q^{28})}
		{q^{5}\psi(q^{38})\, f(-q^{9},-q^{29})},\label{a4.5}\\ 
		\frac{1}{Y_6(q)} \pm Y_6(q)
		&= \frac{\phi(\pm q^{19})\, f(\pm q^{12}, \pm q^{26})}
		{q^{6}\psi(q^{38})\, f(-q^{7},-q^{31})},\label{a4.6}\\ 
		\frac{1}{Y_7(q)} \pm Y_7(q)
		&= \frac{\phi(\pm q^{19})\, f(\pm q^{14}, \pm q^{24})}
		{q^{7}\psi(q^{38})\, f(-q^{5},-q^{33})},\label{a4.7}\\
		\frac{1}{Y_8(q)} \pm Y_8(q)
		&= \frac{\phi(\pm q^{19})\, f(\pm q^{16}, \pm q^{22})}
		{q^{8}\psi(q^{38})\, f(-q^{3},-q^{35})},\label{a4.8}\\ 
		\frac{1}{Y_9(q)} \pm Y_9(q)
		&= \frac{\phi(\pm q^{19})\, f(\pm q^{18}, \pm q^{20})}
		{q^{9}\psi(q^{38})\, f(-q,-q^{37})}.\label{a4.9}
	\end{align}	
	\end{theorem}
	\begin{proof}
		 Replacing $q$ by $q^{\frac{k}{4}}$ and substituting $i=1, 3/2, 2, 5/2, 3, 7/2, 4, 9/2$ and $5$ in \eqref{a1.3}, the proof of Theorem \ref{at4.1} follows similar to that of Theorem \ref{at1.1}. Therefore, we omit the details.
	\end{proof}
	\section{Some Partition-theoretic Results}\label{as5}
	In this section, we obtain color partition identities employing theta-function identities established in Section \ref{as4}.\\
	 \indent To discern the combinatorial interpretation of the theta function identities of $Y_9(q)$, we first define the partition function $p(n)$ to be the number of ways the positive integer $n$ can be written as the sum of positive integers, with the order of these positive integers irrelevant. For example, $p(4)=5$, since there are $5$ ways to write $4$ as sum of positive integers, namely, $(4), (3,1), (2,2), (2,1,1), (1,1,1,1)$. We now observe that a part in a partition of $n$ is said to have $r$ colors if each part has $r$ copies and all of
	them are viewed as distinct objects. For example, if each part of a partition of $3$ has $2$ colors, say yellow and
	blue, then the number of $2$-color partitions of $3$ is $10$, namely,
$(3_y), (3_b), (2_y, 1_y), (2_y, 1_b), (2_b, 1_y), (2_b, 1_b), (1_y, 1_y, 1_y), (1_y, 1_y, 1_b), (1_y, 1_b, 1_b), (1_b, 1_b, 1_b)$.

 For any positive integer $n$ and $r$, let $C_r(n)$ denote the number of partitions of $n$ in which each part has $r$ distinct colors. The generating function of $C_r(n)$ is given by
	\begin{equation*}
		\sum_{n=0}^{\infty} C_r(n) q^n = \frac{1}{(q;q)_\infty^{\,r}} .
	\end{equation*}
	For positive integers $s,m$ and $r$, the quotient $\frac{1}{(q^s;q^m)_\infty^{\,r}}$ is the generating function for the number of partitions of $n$ with parts congruent to $s$ modulo $m$ and
	each part having $r$ colors. For example, the generating function for the number of partitions of a positive integer with parts congruent to $s_1$
	or $s_2$ modulo $m$, each part having $r$ distinct colors is given by
	\begin{equation*}
		\frac{1}{(q^{s_1};q^m)_\infty^{\,r}(q^{s_2};q^m)_\infty^{\,r}}
		=
		\frac{1}{(q^{s_1},q^{s_2};q^m)_\infty^{\,r}},
	\end{equation*}
	and notational convenience, we use $(q^{r\pm};q^t)_\infty := (q^r,q^{t-r};q^t)_\infty$, where $r$ and $t$ are positive integers and $r<t$.\\
	\indent In the following theorems, we offer the combinatorial interpretation of \eqref{a4.9}.
	
	\begin{theorem}\label{at5.1}
		For any integer $n \geq 18$, let $K_1(n)$ denote the number of partitions of $n$ into parts $\equiv \pm 1, \pm 18, \pm 20$ or $\pm 38 \pmod{76}$ such that the parts $\equiv \pm 1$ and $\pm 38 \pmod{76}$ have $2$ colors. Let $K_2(n)$ denote the number of partitions of $n$ into parts 
		$\equiv \pm 18, \pm 20, \pm 37$ or $\pm 38 \pmod{76}$ such that parts $\equiv \pm 33$ and $\pm 38 \pmod{76}$ have $2$ colors. 
		Let $K_3(n)$ denote the number of partitions of $n$ into parts $\equiv \pm 1, \pm 17$ and $\pm 37 \pmod{76}$ with $2$ colors.
		Then
		\[K_1(n) - K_2(n-18) - K_3(n) = 0.\]
	\end{theorem}
	\begin{proof}
		Thanks to \eqref{a2.1} and \eqref{a2.2}, we can rewrite \eqref{a4.9} as
		\begin{equation*}
			\frac{(q^{37\pm}; q^{76})_{\infty}}{q^{9}(q^{1\pm}; q^{76})_{\infty}}
			- q^{9} \frac{(q^{1\pm}; q^{76})_{\infty}}{(q^{37\pm}; q^{76})_{\infty}}
			- \frac{(q^{18\pm}; q^{38})_{\infty}(q^{38\pm}; q^{38})_{\infty}^6}
			{q^{9}(q^{1\pm}; q^{38})_{\infty}(q^{19}; q^{19})_{\infty}^2 (q^{76}; q^{76})_{\infty}^4}
			= 0,
		\end{equation*}
		which can be further simplified to obtain
		\begin{equation*}
			\frac{(q^{37\pm}; q^{76})_{\infty}}{(q^{1\pm}; q^{76})_{\infty}}
			- q^{18} \frac{(q^{1\pm}; q^{76})_{\infty}}{(q^{37\pm}; q^{76})_{\infty}}
			- \frac{(q^{18\pm}, q^{20\pm}; q^{76})_{\infty}(q^{38\pm}; q^{76})_{\infty}^2}
			{(q^{1\pm}, q^{37\pm}; q^{76})_{\infty}(q^{19\pm}; q^{76})_{\infty}^2}
			= 0.
		\end{equation*}
		We now divide the above identity with $(q^{1\pm}, q^{18\pm}, q^{20\pm}, q^{37\pm}; q^{76})_{\infty}
	(q^{38\pm}; q^{76})_{\infty}^2$, to obtain
		\begin{align*}
				&\frac{1}{(q^{18\pm}, q^{20\pm}; q^{76})_{\infty}(q^{1\pm}, q^{38\pm}; q^{76})_{\infty}^2}
				- q^{18} \frac{1}{(q^{18\pm}, q^{20\pm}; q^{76})_{\infty}(q^{37\pm}, q^{38\pm}; q^{76})_{\infty}^2} \\
				&= \frac{1}{(q^{1\pm}, q^{19\pm}, q^{37\pm}; q^{76})_{\infty}^2}.
			\end{align*}
Which is equivalent to
		\begin{equation*}
			\sum_{n=0}^{\infty} K_1(n) q^n
			- q^{18} \sum_{n=0}^{\infty} K_2(n) q^n
			- \sum_{n=0}^{\infty} K_3(n) q^n
			= 0.
		\end{equation*}
		where $K_1(0) = K_2(0) = K_3(0) = 1$.
		Comparing the coefficients of $q^n$ on both sides of the above equation, we complete the proof of Theorem \ref{at5.1}.
	\end{proof}
	\begin{remark}
		For example, by enumerating the relevant partitions of $n = 18$, one can verify that 
		\begin{align*}
			K_1(18) = 20, \quad K_2(0) = 1, \quad K_3(18) = 19,
		\end{align*}
		which satisfy Theorem \ref{at5.1}.
	\end{remark}
	\begin{theorem}\label{at5.2}
		For any integer $n \geq 18$, let $F_1(n)$ denote the number of partitions of $n$ into parts 
		$\equiv \pm 1, \pm 19, \pm 36$ or $\pm 38 \pmod{76}$ such that the parts $\equiv \pm 1$ and $\pm 19 \pmod{76}$ have $2$ colors. 
		Let $F_2(n)$ denote the number of partitions of $n$ into parts 
		$\equiv \pm 19, \pm 36, \pm 37$ or $\pm 38 \pmod{76}$ such that parts $\equiv \pm 19$ and $\pm 37 \pmod{76}$ have $2$ colors. 
		Let $F_3(n)$ denote the number of partitions of $n$ into parts 
		$\equiv \pm 1, \pm 18, \pm 20$ and $\pm 37 \pmod{76}$ such that the parts $\equiv \pm 1$ and $\pm 33 \pmod{68}$ have $2$ colors. 
		Then
		\begin{align*}
			F_1(n) + F_2(n-18) = F_3(n).
		\end{align*}
	\end{theorem}
	\begin{proof}
		Proceeding as in the proof of Theorem \ref{at5.1}, identity \eqref{a4.9} can be expressed as
			\begin{align*}
				&\frac{1}{(q^{36\pm}, q^{38\pm}; q^{76})_{\infty}(q^{1\pm}, q^{19\pm}; q^{76})_{\infty}^2}
				+ q^{18} \frac{1}{(q^{36\pm}, q^{38\pm}; q^{76})_{\infty}(q^{19\pm}, q^{37\pm}; q^{76})_{\infty}^2} \\
				&= \frac{1}{(q^{18\pm}, q^{20\pm}; q^{
				76})_{\infty}(q^{1\pm}, q^{20\pm}; q^{76})_{\infty}^2}.
		\end{align*}
	Noting the generating functions, we rewrite the above identity as
		\begin{equation*}
			\sum_{n=0}^{\infty} F_1(n) q^n
			+ q^{18} \sum_{n=0}^{\infty} F_2(n) q^n
			= \sum_{n=0}^{\infty} F_3(n) q^n,
		\end{equation*}
	where $F_1(0) = F_2(0) = F_3(0) = 1$. Comparing the coefficients of $q^n$ on both sides of the above equation, we complete the proof of Theorem \ref{at5.2}.
	\end{proof}
	\begin{remark}
		For example, by enumerating the relevant partitions of $n = 18$, one can verify that
		\begin{align*}
			F_1(18) = 19, \quad F_2(0) = 1, \quad F_3(18) = 20,
		\end{align*}
		which satisfy Theorem \ref{at5.2}.
	\end{remark}
	\section{Conclusion}
	Vanishing coefficients in arithmetic progression of several $q$-series expansions has been extensively investigated in recent years (see \cite{Andrews, Baruah, Hirschhorn, Laughlin, Laughlin.1, Thejitha}). In the following table, we present the vanishing coefficients in the $q$-series expansions connected with the continued fractions ${Y_n(q)}$, where $n=1,2,3,4,5,6,7,8$ and $9$.
	\[
	\renewcommand{\arraystretch}{2}
	\begin{array}{|l|c|}
		\hline
		\text{$q$-series/continued fractions} & \text{Vanishing coefficients} \\ 
		\hline
		{Y_1^*(q)} = q^{-1} Y_1(q)
		= \frac{(q^{17}, q^{59}; q^{76})_\infty}{(q^{21}, q^{55}; q^{76})_\infty}
		= \sum_{n=0}^{\infty} \alpha_n q^n,
		& \alpha_{38n+35} = 0 \\ 
		\hline
		\frac{1}{{Y}_1^*(q)} = q^{-1} Y_1(q)
		= \frac{(q^{17}, q^{59}; q^{76})_\infty}{(q^{21}, q^{55}; q^{76})_\infty}
		= \sum_{n=0}^{\infty} \alpha'_n q^n,
		& \alpha'_{38n+37} = 0 \\ 
		\hline
		{Y_2^*(q)} = q^{-2} Y_2(q)
		= \frac{(q^{15}, q^{61}; q^{76})_\infty}{(q^{23}, q^{53}; q^{76})_\infty}
		= \sum_{n=0}^{\infty} \beta_n q^n,
		& \beta_{38n+28} = 0 \\ 
		\hline
		\frac{1}{{Y}_2^*(q)} = q^{-2} Y_2(q)
		= \frac{(q^{15}, q^{61}; q^{76})_\infty}{(q^{23}, q^{53}; q^{76})_\infty}
		= \sum_{n=0}^{\infty} \beta'_n q^n,
		& \beta'_{38n+32} = 0 \\ 
		\hline
		{Y_3^*(q)} = q^{-3} Y_3(q)
		= \frac{(q^{13}, q^{63}; q^{76})_\infty}{(q^{25}, q^{51}; q^{76})_\infty}
		= \sum_{n=0}^{\infty} \gamma_n q^n,
		& \gamma_{38n+17} = 0 \\ 
		\hline
		\frac{1}{{Y}_3^*(q)} = q^{-3} Y_3(q)
		= \frac{(q^{13}, q^{63}; q^{76})_\infty}{(q^{25}, q^{51}; q^{76})_\infty}
		= \sum_{n=0}^{\infty} \gamma'_n q^n,
		& \gamma'_{38n+23} = 0 \\ 
		\hline
		{Y_4^*(q)} = q^{-4} Y_4(q)
		= \frac{(q^{11}, q^{65}; q^{76})_\infty}{(q^{27}, q^{49}; q^{76})_\infty}
		= \sum_{n=0}^{\infty} \delta_n q^n,
		& \delta_{38n+2} = 0 \\ 
		\hline
		\frac{1}{{Y}_4^*(q)} = q^{-4} Y_4(q)
		= \frac{(q^{11}, q^{65}; q^{76})_\infty}{(q^{27}, q^{49}; q^{76})_\infty}
		= \sum_{n=0}^{\infty} \delta'_n q^n,
		& \delta'_{38n+10} = 0 \\ 
		\hline
		{Y_5^*(q)} = q^{-5} Y_5(q)
		= \frac{(q^{9}, q^{67}; q^{76})_\infty}{(q^{29}, q^{47}; q^{76})_\infty}
		= \sum_{n=0}^{\infty} \zeta_n q^n,
		& \zeta_{38n+21} = 0 \\ 
		\hline
		\frac{1}{{Y}_5^*(q)} = q^{-5} Y_5(q)
		= \frac{(q^{9}, q^{67}; q^{76})_\infty}{(q^{29}, q^{47}; q^{76})_\infty}
		= \sum_{n=0}^{\infty} \zeta'_n q^n,
		& \zeta'_{38n+33} = 0 \\ 
		\hline
		{Y_6^*(q)} = q^{-6} Y_6(q)
		= \frac{(q^{7}, q^{69}; q^{76})_\infty}{(q^{31}, q^{45}; q^{76})_\infty}
		= \sum_{n=0}^{\infty} \eta_n q^n,
		& \eta_{38n+36} = 0 \\ 
		\hline
		{Y_7^*(q)} = q^{-7} Y_7(q)
		= \frac{(q^{5}, q^{71}; q^{76})_\infty}{(q^{33}, q^{43}; q^{76})_\infty}
		= \sum_{n=0}^{\infty} \xi_n q^n,
		& \xi_{38n+9} = 0 \\ 
		\hline
		{Y_8^*(q)} = q^{-8} Y_8(q)
		= \frac{(q^{3}, q^{73}; q^{76})_\infty}{(q^{35}, q^{41}; q^{76})_\infty}
		= \sum_{n=0}^{\infty} \lambda_n q^n,
		& \lambda_{38n+16} = 0 \\ 
		\hline
		{Y_9^*(q)} = q^{-9} Y_9(q)
		= \frac{(q, q^{75}; q^{76})_\infty}{(q^{37}, q^{39}; q^{76})_\infty}
		= \sum_{n=0}^{\infty} \sigma_n q^n,
		& \sigma_{38n+19} = 0 \\ 
		\hline
		\hline
	\end{array}
	\]	
	The vanishing coefficient results mentioned in the above table can be derived following a proof similar to that of Theorems 3.1 and 3.2 in \cite{D.S}. We encourage interested readers to further investigate vanishing coefficients for continued fractions of other even order.

	\bigskip
	\bigskip
	
	\noindent
	Department of Mathematics\\
	Ramanujan School of Mathematical Sciences\\
	Pondicherry University\\
	Puducherry- 605 014, India.\\

	\noindent Email: \texttt{iamdipikasarkar@pondiuni.ac.in}\\
	\noindent Email: \texttt{dr.fathima.sn@pondiuni.ac.in} (\Letter)\\
	\noindent Email: \texttt{tthejithamp@pondiuni.ac.in}
	

\begin{thebibliography}{9}
		\bibitem{Andrews}
		G. E. Andrews, D. Bressoud,
		\textit{Vanishing coefficients in infnite product expansion},
		J. Aust.Math. Soc. Ser. 27 (1979) 199-202.
		\bibitem{Baruah}
		N. D. Baruah and M. Kaur,
		\textit{Some results on vanishing coefficients in infinite product expansions},
		Ramanujan J. 53 (2020), 551–568.
		
		\bibitem{Berndt}
		B. C. Berndt,
		\textit{Ramanujan's Notebooks, Part III},
		Springer-Verlag, New York, 1991.
		\bibitem{Hirschhorn}
		M. D. Hirschhorn, 
		\textit{Two remarkable q-series expansions},
		Ramanujan J. 49(2) (2018) 451-463.
		
		\bibitem{Laughlin}
		M. C. Laughlin,
		\textit{Further results on vanishing coefficients in infinite product expansions},
		J. Aust. Math. Soc. Ser. A {98} (2015), 69-77.
			\bibitem{Laughlin.1}
		J. M. Laughlin,
		\textit{Some observations on Lambert series, vanishing coefficients and dissections of infinite products and series}, https://doi.org/10.48550/arXiv.1906.11978.
		
		
		\bibitem{Saikia.1}
		S. Rajkhowa and N. Saikia,
		\textit{Some Identities of Ramanujan's $q$-continued fraction of order fourteen and twenty-eight, and vanishing coefficients},
		Functiones et Approximatio,Commentarii Mathematici, vol. 70, no. 2, 2024.
		\bibitem{Saikia.4}
		S. Rajkhowa and N. Saikia,
		\textit{Theta-function identities, explicit values for Ramanujan's continued fractions of order sixteen and applications to partition theory.},
		ntegers: Electronic Journal of Combinatorial Number Theory, 23, A59 (2023).
		\bibitem{Saikia.2}
		S. Rajkhowa and N. Saikia,
		\textit{S. Rajkhowa and N. Saikia, Some results on Ramanujan’s continued fractions of order ten and applications},
		Indian J. Pure Appl. Math. 56(1) (2025),13–37.
		
		\bibitem{Saikia.3}
		S. Rajkhowa and N. Saikia,
		\textit{Theta function identities and vanishing coefficient results for continued fractions of order sixty-four},
		Indian Journal of Pure and Applied Mathematics, Springer, vol. 57(2), pages 585-596, April (2026).
		\bibitem{Raksha}
		Raksha and B. R. Srivatsa Kumar,
		\textit{Some identities of Ramanujan's q-Continued Fraction of Order Eighteen, Twenty-Six and Thirty, and Vanishing Coefficients},
		https://doi.org/10.48550/arXiv.2311.06298
		\bibitem{R.Notebooks (2 volumes)}
		S. Ramanujan,
		\textit{Notebooks (2 volumes)},
		Tata Institute of Fundamental Research, Bombay, 1957.
		\bibitem{D.S}
		D. Sarkar and S. N. Fathima,
		\textit{On Ramanujan's $q$-Continued Fractions of Order Thirty-Four and Sixty-Eight},
		https://doi.org/10.48550/arXiv.2605.28841
		\bibitem{Thejitha}
		Thejitha, M. P., Anand, A., and Fathima, S. N. (2025). \textit{Vanishing Coefficients of $q^{5n+ r}$ and $q^{7n+ r}$ in Certain Infinite q-series Expansions.} arXiv preprint arXiv:2510.05929.
	\end{thebibliography}
\end{document}